\documentclass[11 pt]{amsart}
\usepackage{amssymb}
\usepackage{amsmath}
\usepackage{amscd}
\usepackage{graphicx}
\usepackage{latexsym}

\calclayout\evensidemargin .6in

\def\R{\mathbb{R}}

\newcommand{\M}{\mathcal{M}(\Sigma,\phi)}

\newcommand{\reeb}{\mathcal{S}_{\textrm{Reeb}}}
\newcommand{\dd}{\mathcal{D}}

\newtheorem{theorem}{Theorem}[section]
\newtheorem{lemma}[theorem]{Lemma}

\newtheorem{corollary}[theorem]{Corollary}

\theoremstyle{definition}

\newtheorem{example}[theorem]{Example}
\newtheorem{remark}[theorem]{Remark}
\newtheorem{definition}[theorem]{Definition}

\title[Tight Contact Structures on Contact Mapping Tori and their Folded Sums]
{Tight Contact Structures on Contact Mapping Tori \\ and their Folded Sums}

\author{Mehmet Firat Arikan}
\address{Dept. of Mathematics, Middle East Technical University, Ankara, TURKEY}
\email{farikan@metu.edu.tr}


\begin{document}
\begin{abstract}
It is known that the folded sum of two contact mapping tori whose fibers are compact exact symplectic manifolds having a common convex boundary (called the ``fold'') admits a cooriented contact structure compatible with the obvious fibration map onto the circle. Here we first provide an alternative bundle-theoretical construction of such a ``folded'' contact structure based on a gluing process near the fold. Moreover, we prove that in any odd dimension $2n+1\geq 7$ a folded contact structure on a folded sum of two contact mapping tori is tight if the induced contact form on the (common) contact fold admits no contractible Reeb orbit. In particular, any contact mapping torus of an odd dimension $2n+1\geq 7$ is tight if the induced contact form on the convex boundary of a fiber admits no contractible Reeb orbit.

\end{abstract}

\maketitle

%
%

\section{Introduction}
In the literature, the ``folded sum operation'' (Definition \ref{def:Folded_Sum_of_Mapping_Tori}) can be found under the name ``blown-up summed open book'' (see \cite{W}) whose particular form (considered in the present paper) can also be seen as a ``contact fiber connected sum'' of two open books along their common bindings (see \cite{Ge}).

\medskip
For any smooth manifold $\Sigma$ and a chosen self-diffeomorphism $\phi \in \textrm{Diff}(\Sigma)$, the \textit{mapping torus} $\M$ is the quotient manifold of $\Sigma \times [0,1]$ obtained by identifying each point $x \in \Sigma \times \{0\}$ with its image $\phi(x) \in \Sigma \times \{1\}$. (If $\Sigma$ is orientable, so is $\M$ provided $\phi \in \textrm{Diff}\,^+(\Sigma)$.) By its definition, $\M$ fibers over $S^1$ and $\phi$ is the \textit{monodromy}. Equivalently, in bundle theoretical words, there is a (locally trivial) fiber bundle $\pi: \M \longrightarrow S^1$ with fibers $\Sigma$ and the structure group $\textrm{Diff}(\Sigma)$.

\medskip
If we require $\phi \in \textrm{Diff}\,^+(\Sigma)$ to be identity near $\partial \Sigma$ (i.e., $\phi \in \textrm{Diff}\,^+(\Sigma,\partial \Sigma)$), then one can glue $\M$ and $\partial \Sigma \times D^2$ along their boundaries (so that the angular coordinate on $D^2$ agrees with the map $\pi$) to obtain a closed orientable manifold
\begin{center}
$Y(\Sigma,\phi)=\M \cup_{\partial}(\partial \Sigma \times D^2).$
\end{center}
Such a decomposition is called an (\textit{abstract}) \textit{open book} of the resulting manifold $Y(\Sigma,\phi)$ and will be denoted by $\mathcal{OB}(\Sigma,\phi)$.

\medskip
Throughout this note, we will restrict our attention to the case where the fiber $(\Sigma,d\beta)$ is a compact exact symplectic manifold with convex (so contact) boundary $(\partial \Sigma, \textrm{Ker}(\beta|_{T\partial \Sigma}))$ and $\phi\in \textrm{Exact}(\Sigma,\partial \Sigma, d\beta)\subset \textrm{Diff}\,^+(\Sigma,\partial \Sigma)$ is an exact symplectomorphism which is identity near $\partial \Sigma$. Such a mapping torus will be called a \textit{contact mapping torus}.

\medskip
If $\M$ is a contact mapping torus, then $\mathcal{OB}(\Sigma,\phi)$ is a \textit{contact open book}. That is, one can construct a contact form on $\M$ and glue with the natural one on $\partial \Sigma \times D^2$ to obtain a contact form on $Y(\Sigma,\phi)$ (see \cite{TW}, \cite{Gi, Gi2}). Their construction of the contact form on $\M$ is based on the fact that the monodromy is an exact symplectomorphism of the fiber (which is identity near the boundary), and also that the identification map building $\M$ can be chosen in such a way that it pulls back the product contact form on $\Sigma \times [0,1]$ to itself.

\medskip
In this note, we will construct contact forms on contact mapping tori by making use of their (exact) symplectic bundle structures. As pointed out in Remark \ref{rmk:new_and_old_structures_isotopic}, the resulting contact structures will be contactomorphic to the ones constructed in \cite{TW, Gi, Gi2}. Contact mapping tori are the simplest compact contact symplectic fibrations which are studied in \cite{Ari} where a natural bundle theoretical method constructing contact forms on them is given. Such contact structures are called \textit{bundle contact structures}. This bundle theoretical construction has more advantages in constructing and studying contact forms on folded sums of contact mapping tori. For instance, it is easier to adapt it to the technique given in \cite{GS} where contact forms on products of closed $4$-manifolds with $S^1$ are constructed.

\medskip
Our main existence and tightness results for folded sums of contact mapping tori can be stated as follows (see Section \ref{sec:Contact_Structures_on_Folded_Sums} and Section \ref{sec:Tight_Contact_Structures_on_Folded_Sums} for the precise statements): 

\begin{theorem} \label{thm:Contact_Structure_on_Folded_Sum_0}
For $n\geq 1$, let $\pi_i: \mathcal{M}(\Sigma^{2n}_i,\phi_i) \longrightarrow S^1$ be any two contact mapping tori with the structure groups $\emph{Exact}(\Sigma_i,\partial \Sigma_i, d\beta_i)$ such that $(\Sigma^{2n}_1, d\beta_1)$ and $(\Sigma^{2n}_2, d\beta_2)$ have the same convex boundary. Then the total space $\mathcal{M}$ of their ``folded sum'' admits a (co-oriented) contact structure $\eta=\emph{Ker}(\sigma)$ which is compatible with the (naturally arising) fibration map $\pi:\mathcal{M} \to S^1$, i.e., $d\sigma$ restricts to a folded exact symplectic form on every fiber.
\end{theorem}

A contact structure $\eta$ on a folded sum (constructed as in the above statement) will be called ``\textit{folded}''. But one should note that the defining ``\textit{folded}'' contact form $\sigma$ is an honest (smooth) contact form. 

\begin{theorem} \label{thm:Contact_Structure_on_Folded}
For $n\geq 3$, let $\pi_i: \mathcal{M}(\Sigma^{2n}_i,\phi_i) \longrightarrow S^1$ be any two contact mapping tori with the structure groups $\emph{Exact}(\Sigma_i,\partial \Sigma_i, d\beta_i)$ where $(\Sigma^{2n}_1, d\beta_1)$ and $(\Sigma^{2n}_2, d\beta_2)$ are Weinstein domains with the same convex boundary $Y$. Suppose that the coinciding induced contact forms $\beta_1|_{TY}=\beta_2|_{TY}$ on $Y$ have no contractible periodic Reeb orbit. Then the folded contact structure $\eta$ on the total space of their folded sum (constructed as in Theorem \ref{thm:Contact_Structure_on_Folded_Sum_0}) is tight.
\end{theorem}

One should note that the ideas and arguments used in the proof of the above theorems can be extended to obtain similar results for the ``folded multi-sum'' of any number of contact mapping tori (which are glued along their contactomorphic boundary components). As an immediate consequence of the last theorem, we have 

\begin{corollary} \label{thm:Tight_Contact_Structure_on_mapping_Tori}
	For $n\geq 3$, let $\pi: \mathcal{M}(\Sigma^{2n},\phi) \longrightarrow S^1$ be any contact mapping torus with the structure group $\emph{Exact}(\Sigma,\partial \Sigma, d\beta)$ where $(\Sigma^{2n}, d\beta)$ is a Weinstein domain with the convex boundary $(Y=\partial \Sigma,\emph{Ker}(\beta|_{TY}))$. If the induced contact form $\beta|_{TY}$ has no contractible periodic Reeb orbit, then the bundle contact structure $\eta$ on $\mathcal{M}(\Sigma^{2n},\phi)$ (constructed as in Theorem \ref{thm:Contact_form_mapping _tori}) is tight. \qed
\end{corollary}

%
%

\section{Bundle Contact Structures on Contact Mapping Tori} \label{sec:Contact_Structures_on_Mapping_Tori}

Any contact mapping torus admits a structure of compact contact symplectic fibration over $S^1$ \cite{Ari}. Indeed, one can construct ``compatible'' contact structures on such fibrations using their bundle structures. In precise words, we have:

\begin{definition} [\cite{Ari}] \label{def:compatibility}
Let $\pi:\M \to S^1$ be a mapping torus with $\phi \in\textrm{Diff}\,^+(\Sigma,\partial \Sigma)$, and suppose that $\M$ admits a co-oriented contact structure $\eta$. Then $\eta$ is said to be \emph{compatible with} $\pi$ if there is a contact form $\sigma$ for $\eta$ such that $d\sigma$ restricts to an exact symplectic form on every fiber of $\pi$.
\end{definition}

\begin{theorem} \label{thm:Contact_form_mapping _tori}
For $n\geq 1$, let $\pi:\mathcal{M}(\Sigma^{2n},\phi) \longrightarrow S^1$ be a contact mapping torus with the structure group $\emph{Exact}(\Sigma,\partial \Sigma, d\beta)$. Then the total space $\M$ admits a (co-oriented) contact structure $\eta=\emph{Ker}(\sigma)$ which is compatible with $\pi$, that is, $d\sigma$ restricts to an exact symplectic form on every fiber. Moreover, the contact form $\sigma$ is equal to a product form on some collar neighborhood of the boundary $\partial \M\cong \partial \Sigma \times S^1$.
\end{theorem}

\begin{proof}
Let $(\Sigma,d\beta)$ be any fixed fiber of $\pi$. Then the transition maps are contained in the structure group $\textrm{Exact}(\Sigma,\partial \Sigma, d\beta)$. Since $\pi:\M \longrightarrow S^1$ is locally trivial fiber bundle, $\M$ is constructed by patching the trivial $\Sigma$-bundles over open arcs in $S^1$. Therefore, there is an open cover $\{U_\rho\}_\rho$ of $S^1$ and diffeomorphisms $\varphi_\rho: \pi^{-1}(U_\rho) \to U_\rho \times \Sigma$ satisfying $$\pi|_{\pi^{-1}(U_\rho)}= p_1 \circ \varphi_\rho$$
(where $p_1:U_\rho \times \Sigma \longrightarrow U_\rho$ is the projection) such that the collection $\{U_\rho,\varphi_\rho\}_\rho$ trivializes $\pi$. For each $\rho$ by restricting $\varphi_\rho$ to each fiber $\Sigma_\theta:=\pi^{-1}(\theta)$ over $U_\rho$ and then projecting onto $\Sigma$-factor, we obtain a smooth map $\varphi_\rho(\theta):\Sigma_\theta \to \Sigma$, and as a result, a pullback form 
\begin{center}
$\lambda_\theta:=(\varphi_\rho(\theta))^*(\beta).$
\end{center}
Note that $(\Sigma_\theta,d\lambda_\theta)$ is a compact exact symplectic manifold for each $\theta \in U_\rho \subset S^1$. Also using the projection $p_2:U_\rho \times \Sigma \longrightarrow \Sigma$ onto the second factor, we have the pullback form 
\begin{center}
$\lambda_\rho:=p_2^*(\beta) \in \Omega^1(U_\rho \times \Sigma).$
\end{center}
Hence, we obtain a $1$-form $\phi_\rho^*(\lambda_\rho) \in \Omega^1(\pi^{-1}(U_\rho))$. By construction, $\phi_\rho^*(\lambda_\rho)|_{\Sigma_\theta}=\lambda_\theta$ for each $\theta \in U_\rho$. Let $f_\rho:S^1 \to [0,1]$ be a partition of unity subordinating to $\{U_\rho\}_\rho$ of $S^1$. Set
$$\lambda_\beta:=\sum_{\rho}(f_\rho \circ \pi)\varphi_\rho^*(\lambda_\rho) \;\in \Omega^1(\M).$$
If $\iota_\theta: \Sigma_\theta \longrightarrow \M$ is the inclusion map of the fiber over $\theta$, then we compute
$$\iota_\theta^*(\lambda_\beta)=\iota_\theta^*\left(\sum_{\rho}(f_\rho \circ \pi)\varphi_\rho^*(\lambda_\rho)\right)=
										  \sum_{\rho}(f_\rho \circ \pi)\iota_\theta^*(f_\rho^*(\lambda_\rho))=\lambda_\theta,$$
that is, $\lambda_\beta$ restricts to $\lambda_\theta$ on each fiber $\Sigma_\theta$, and so $d\lambda_\beta=d\lambda_\theta$ is nondegenerate on each $\Sigma_\theta$. In fact, $d\lambda_\beta\in \Omega^2(\M)$ is nondegenerate (symplectic) on all vertical subspaces
$$\textrm{Vert}_x:=\textrm{Ker}(D\pi(x))=T_x\Sigma_\theta\subset T_x\M, \quad \textrm{for } \theta \in \Sigma, \textrm{ and for } x \in \pi^{-1}(\theta)$$
where $D\pi(x)$ denotes the derivative map of $\pi$ at $x$.
Therefore, at each $x\in \M$, we have the horizontal space (the symplectic orthogonal complement of $\textrm{Vert}_x$ in $T_x\M$):
$$\textrm{Hor}_x:=\{u\in T_x\M\,:\, d\lambda_\beta(u,v)=0, \; \forall v\in \textrm{Vert}_x\}\subset T_x\M.$$
At each $x\in \Sigma_\theta$, since the restriction $d\lambda_\theta=d\lambda_\beta|_{\Sigma_\theta}$ is of full rank and the derivative map $D\pi(x):T_x\M=\textrm{Vert}_x \oplus \textrm{Hor}_x \to T_\theta S^1$ vanishes on $\textrm{Vert}_x$, we conclude that $D\pi(x)|_{\textrm{Hor}_x}:\textrm{Hor}_x \to T_\theta S^1$ is an isomorphism. Now for any $K>0$, consider the $1$-form
$\sigma\in \Omega^1(\M)$ defined by
$$\sigma:=\lambda_\beta+K\pi^*(d\theta).$$
Then noting that $\textrm{dim}(\Sigma)=2n$, we compute
\begin{eqnarray*}
\sigma \wedge (d\sigma)^{n}&=& (\lambda_\beta +K\pi^*(d\theta)) \wedge (d\lambda_\beta +K\pi^*(d^2\theta))^{n}=(\lambda_\beta +K\pi^*(d\theta)) \wedge (d\lambda_\beta )^{n} \\
&=& \lambda_\beta \wedge (d\lambda_\beta)^{n}+ K \pi^*(d\theta) \wedge (d\lambda_\beta)^n.
\end{eqnarray*}
Observe that at each $x \in \M$, the form $\sigma \wedge (d\sigma)^{n}|_x$ is a positive volume form on the tangent space $T_x\M$ provided $K=K_x>0$ is chosen large enough. Since $\M$ is compact, there exists a sufficiently large $K>0$ such that $\sigma \wedge (d\sigma)^{n}|_x$ is a positive volume form on $T_x\M$ for all $x\in \M$. Equivalently, for any large enough  $K>0$, $\sigma$ is a global contact form and defines a (co-oriented) contact structure $\eta=\textrm{Ker}(\sigma)$ on $\M$.

Moreover, for each $\theta\in S^1$ we have
$$\iota_\theta^*(\sigma)=\iota_\theta^*(\lambda_\beta +K\pi^*(d\theta))=\iota_\theta^*(\lambda_\beta)+ K(\pi \circ \iota_\theta)^*(d\theta)=\iota_\theta^*(\lambda_\beta)+0=\lambda_\theta$$ which means, by definition, that $\eta$ is compatible with $\pi$ as required.

\medskip
Finally, observe that since the structure group of $\pi:\M \to S^1$ is $\textrm{Exact}(\Sigma,\partial \Sigma, d\beta)$ (which is a subgroup of $\textrm{Diff}\,^+(\Sigma,\partial \Sigma)$) and $S^1$ is compact, there exists $\epsilon>0$ such that a collar neighborhood $N$ of $\partial \M \cong \partial \Sigma \times S^1$ has an identification $N=(-\epsilon,0] \times \partial \Sigma \times S^1$. Moreover, the $1$-form $\lambda_\beta$ is constructed in such a way that it restricts to $\lambda_\theta$ on each fiber $\Sigma_\theta$. But all these $\lambda_\theta$'s are obtained by pulling back the same $\beta$ fixed at the beginning. Hence, on each collar neighborhood $N \cap \Sigma_\theta$, $\lambda_\theta$ (and so $\lambda_\beta$) restricts to $\beta$. Equivalently, on the collar neighborhood $N$, $\sigma$ is the product form $\beta+Kd\theta$ as claimed.

\end{proof}

\begin{definition}[\cite{Ari}] The contact form $\sigma$ on $\mathcal{M}(\Sigma,\phi)$ constructed in the above proof will be called a \textit{bundle contact form}, and its kernel $\eta$ will be called a \textit{bundle contact structure}.
\end{definition}

\begin{remark} \label{rmk:new_and_old_structures_isotopic}
One can see that the mapping torus $\pi:\M\longrightarrow S^1$ together with the bundle contact structure $\eta=\textrm{Ker}(\sigma)$ is, in fact, contactomorphic to the contact manifold $(\mathcal{M}'(\Sigma,\phi),\eta'=\textrm{Ker}(\sigma'))$ constructed in \cite{TW,Gi,Gi2} where $$\mathcal{M}'(\Sigma,\phi)=\Sigma \times [0,1] / \sim \; ,\qquad (x,0)\sim (\phi(x),1)$$ and the identification map of ``$\sim$'' is directly used in the construction of $\sigma'$. 

\medskip
To verify this claim, first close up the contact manifolds $(\M,\eta), (\mathcal{M}'(\Sigma,\phi),\eta')$ by gluing two copies of $\partial \Sigma \times D^2$ equipped with suitable contact structures which can be smoothly glued to $\eta$ and $\eta'$, respectively. This is just the standard procedure in obtaining contact open books from contact mapping tori used in \cite{TW,Gi,Gi2}. Therefore, we obtain two isomorphic contact open books $\mathcal{OB}:=\mathcal{OB}(\Sigma,\phi)$, $\mathcal{OB}':=\mathcal{OB}'(\Sigma,\phi)$ on the closed contact manifolds, say $$(M,\xi=\textrm{Ker}(\alpha)),\;\; (M',\xi'=\textrm{Ker}(\alpha')),$$ respectively, where $\alpha$ (resp. $\alpha'$) is obtained from $\sigma$ (resp. $\sigma'$) by (smoothly) gluing a suitable contact form on $\partial \Sigma \times D^2$. Note that the open books $\mathcal{OB},\mathcal{OB}'$ support $\xi,\xi'$, respectively, and since they are isomorphic, there exists a fiber preserving diffeomorphism $f:M \to M'$, that is, taking the fibers of $\mathcal{OB}$ onto those of $\mathcal{OB}'$ and mapping any (trivial) neighborhood of the binding of $\mathcal{OB}$ to that of $\mathcal{OB}'$. Observe that the pullback form $\alpha^*:=f^*(\alpha')$ defines a contact structure $\xi^*=\textrm{Ker}(\alpha^*)$ which is also supported by $\mathcal{OB}$. Therefore, by Giroux correspondence, there exists an isotopy $\{\xi_t | t\in [0,1]\}$ of contact structures on $M$ between $\xi_0=\xi$ and $\xi_1=\xi^*$. In particular, there is a contactomorphism $$g:(M,\xi)\to (M,\xi^*).$$ Moreover, by the construction of $\xi^*$, the map $f$ is a contactomorphism from $(M,\xi^*)$ onto $(M',\xi')$, so their composition defines a contactomorphism $f\circ g:(M,\xi)\to (M',\xi')$. Finally, by considering the restriction of $g$ between the complements of the suitable neighborhoods of the bindings of $\mathcal{OB}$ and $\mathcal{OB}'$, one can obtain the desired contactomorphism between $(\M,\eta)$ and $(\mathcal{M}'(\Sigma,\phi),\eta')$. \qed
\end{remark}

%
%

\section{Folded Contact Structures on Folded Sums of Mapping Tori} \label{sec:Contact_Structures_on_Folded_Sums}

For a given mapping torus $\pi:\M \longrightarrow S^1$ with $\phi \in \textrm{Diff}\,^+(\Sigma,\partial \Sigma)$, the boundary  $$\partial \M=\bigcup_{\theta \in S^1} \partial \pi^{-1}(\theta)\cong \partial \Sigma \times S^1$$ has a trivial neighborhood which is diffeomorphic to a product $(-\epsilon, 0] \times \partial \Sigma \times S^1.$ This makes the following construction possible where $\overline{\Sigma}$ denotes $\Sigma$ with opposite orientation.

\begin{definition} \label{def:Folded_Sum_of_Mapping_Tori}
Let $\pi_i:\mathcal{M}(\Sigma_i,\phi_i) \longrightarrow S^1$ \ (for $i=1,2$) be two mapping tori with $\phi_i \in \textrm{Diff}\,^+(\Sigma_i,\partial \Sigma_i)$ such that $\partial \Sigma_1=Y=\partial \Sigma_2$. Then the \textit{folded sum} of $\pi_1$ and $\pi_2$ is a smooth fiber bundle $\pi: \mathcal{M}(\Sigma_1\cup_Y \overline{\Sigma}_2,\tilde{\phi}_1 \circ \tilde{\phi}_2) \longrightarrow S^1$ constructed as follows:
\begin{itemize}
\item[(i)] The total space of $\pi$ is a smooth manifold with the identification
\begin{center}
$\mathcal{M}(\Sigma_1\cup_Y \overline{\Sigma}_2,\tilde{\phi}_1 \circ \tilde{\phi}_2)=(\Sigma_1\cup_Y \overline{\Sigma}_2) \times [0,1] / \sim$
\end{center}
where $(x,0)\sim (\tilde{\phi}_1(x),1)$ for any $x\in \Sigma_1$ and $(y,0)\sim (\tilde{\phi}_2(y),1)$ for any $y\in \Sigma_2$. Here $\tilde{\phi}_1$ (resp. $\tilde{\phi}_2$) is the extension of $\phi_1$ (resp. $\phi_2$) via identity over $\Sigma_2$ (resp. $\Sigma_1$).

\item[(ii)] The fiber of $\pi$ over $\theta \in S^1$ is the boundary sum of the fiber of $\pi_1$ over $\theta$ with the corresponding one of $\pi_2$ with the orientation reversed, that is,
\begin{center}
$\pi^{-1}(\theta)=\pi_1^{-1}(\theta)\cup_{\partial}\overline{\pi_2^{-1}(\theta)}\cong \Sigma_1\cup_Y \overline{\Sigma}_2.$
\end{center}
\item[(iii)] The monodromy of $\pi$ is the composition $\tilde{\phi}_1 \circ \tilde{\phi}_2$ which lies in the group
\begin{center}
$\textrm{Diff}\,^+(\Sigma_1\cup_Y \overline{\Sigma}_2,Y):=\{\phi\in \textrm{Diff}\,^+(\Sigma_1\cup_Y \overline{\Sigma}_2) \;| \; \phi \,\textrm{ is identity near }\, Y \}$
\end{center}
which is, indeed, the structure group of $\pi$.
\end{itemize}
\end{definition}

If we further assume that $\Sigma_1, \Sigma_2$ are compact exact symplectic manifolds with coinciding convex boundaries, and if each $\phi_i$ is chosen to be exact symplectomorphism of $\Sigma_i$ (which is identity near boundary), then one can construct (see Theorem \ref{thm:Contact_Structure_on_Folded_Sum}) a contact structure $\eta$ on the folded sum (defined as in Definition \ref{def:Folded_Sum_of_Mapping_Tori}). The contact form, say $\sigma$, on the folded sum (defining $\eta$) is obtained by gluing the ones on the contact mapping tori $\mathcal{M}(\Sigma_i,\phi_i)$ constructed in the proof of Theorem \ref{thm:Contact_form_mapping _tori}. The resulting contact structure is ``compatible'' with the fiber bundle structure on the folded sum in the sense that $d\sigma$ defines a \textit{folded symplectic structure} when restricted to each fiber $\Sigma_1 \cup_{\partial}\overline{\Sigma}_2$. More precisely,

\begin{definition} \label{def:compatibility_folded}
Let $\pi: \mathcal{M}(\Sigma_1\cup_Y \overline{\Sigma}_2,\tilde{\phi}_1 \circ \tilde{\phi}_2) \longrightarrow S^1$ be any folded sum of two mapping tori (as in Definition \ref{def:Folded_Sum_of_Mapping_Tori}) such that $\mathcal{M}(\Sigma_1\cup_Y \overline{\Sigma}_2,\tilde{\phi}_1 \circ \tilde{\phi}_2)$ admits a co-oriented contact structure $\eta$. Then $\eta$ is said to be \emph{compatible with} $\pi$ if there exists a contact form $\sigma$ for $\eta$ such that $d\sigma$ restricts to a folded exact symplectic form on every fiber of $\pi$.
\end{definition}

\begin{theorem} \label{thm:Contact_Structure_on_Folded_Sum}
For $n\geq 1$, let $\pi: \mathcal{M}(\Sigma_1\cup_Y \overline{\Sigma}_2,\tilde{\phi}_1 \circ \tilde{\phi}_2) \longrightarrow S^1$ be any folded sum of two contact mapping tori with the structure groups $\emph{Exact}(\Sigma_i,\partial \Sigma_i, d\beta_i)$ such that $(\Sigma^{2n}_1, d\beta_1)$ and $(\Sigma^{2n}_2, d\beta_2)$ have the same convex boundary $Y$. Then the total space $\mathcal{M}(\Sigma_1\cup_Y \overline{\Sigma}_2,\tilde{\phi}_1 \circ \tilde{\phi}_2)$ admits a (co-oriented) contact structure $\eta=\emph{Ker}(\sigma)$ which is compatible with $\pi$.
\end{theorem}

\begin{proof} Let $\pi_i:\mathcal{M}(\Sigma_i,\phi_i) \longrightarrow S^1$ \ (for $i=1,2$) be the contact mapping tori which construct $\pi$. Fix a fiber $(\Sigma_i, d\beta_i)$ of $\pi_i$ for each $i$. By assumption $(\Sigma_1, d\beta_1)$ and $(\Sigma_2, d\beta_2)$ are both exact symplectic fillings of their common convex boundary, say $(Y,\xi)$. Without loss of generality, assume $\xi=\textrm{Ker}(\alpha)$ where $\beta_1|_{TY}=\alpha=\beta_2|_{TY}$, and so for each $i$ we have $\beta_i=e^t\alpha$ on some collar neighborhood $(-\epsilon,0]\times Y$ of $Y=\partial \Sigma_i$ in $\Sigma_i$.

\medskip
By Theorem \ref{thm:Contact_form_mapping _tori}, there exists a bundle contact structure $\eta_i=\textrm{Ker}(\sigma_i)$ on each total space $\mathcal{M}(\Sigma_i,\phi_i)$ compatible with $\pi_i$. The bundle contact form defining $\eta_i$ is given by
$$\sigma_i=\lambda_{\beta_i}+K_i\pi_i^*(d\theta)\in \Omega^1(\mathcal{M}(\Sigma_i,\phi_i))\quad \textrm{for }\, K_i>0 \,\textrm{ large enough},$$ and it is a product form on a sufficiently small collar neighborhood of the boundary $Y \times S^1$ of $\mathcal{M}(\Sigma_i,\phi_i)$. More precisely, there exists $\epsilon_i>0$ such that
$$\sigma_i=\beta_i+K_i\,d\theta \quad \textrm{on }\, (-1-\epsilon_i,-1] \times Y \times S^1.$$
(Here we intentionally make a $1$ unit shift in the Liouville direction so that $\partial \Sigma_i=\{-1\}\times Y$ for each $i$. This will provide an easier formulation of the gluing arguments below.)
Set
\begin{center}
$K=\textrm{max}(K_1,K_2)$ \quad and \quad $\epsilon=\textrm{min}(\epsilon_1,\epsilon_2)$.
\end{center}
Then for each $i$, $\sigma_i=\lambda_{\beta_i}+K\pi_i^*(d\theta)$ is a bundle contact form on $\mathcal{M}(\Sigma_i,\phi_i)$ and equal to the product form $\sigma_i=\beta_i+Kd\theta$ on $(-1-\epsilon,-1] \times Y \times S^1$. Also note that replacing $\alpha$ with $(1/K)\alpha$ (they define the same contact structure on $Y$), and after shifting along the $t$-direction, we may take
\begin{center}
$\beta_i=Ke^{t+1}\alpha$ \;\; on \;\; $(-1-\epsilon,-1] \times Y$ \;\;for each $i$.
\end{center}
Now, up to diffomorphism, we can consider the total space of $\pi$ as follows:
\begin{center}
$\displaystyle \mathcal{M}(\Sigma_1\cup_Y \overline{\Sigma}_2,\tilde{\phi}_1 \circ \tilde{\phi}_2)=\mathcal{M}(\Sigma_1,\phi_1)\bigcup_{\{-1\}\times Y \times S^1} [-1,1] \times Y \times S^1 \bigcup_{\{1\}\times Y \times S^1} \overline{\mathcal{M}(\Sigma_2,\phi_2)}$.
\end{center}
\noindent Choose two smooth functions $f,g:(-1-\epsilon,1+\epsilon) \to \mathbb{R}$ satisfying:

\begin{itemize}
\item[(1)] $f$ is an even function such that $f(t)=Ke^{t+1}$ near $(-1-\epsilon,-1]$,
\item[(2)] $g$ is an odd function such that $g(t)=K$ near $(-1-\epsilon,-1]$,
\item[(3)] $f'g-fg'>0$ everywhere on $(-1-\epsilon,1+\epsilon)$,
\item[(4)] $f'(t)=0$ if and only if $t=0$.
\end{itemize}
(See Figure \ref{fig:Gluing_functions} for the graphs of such $f$ and $g$.)

\begin{figure}[h]
	\centering
	\includegraphics[scale=.81]{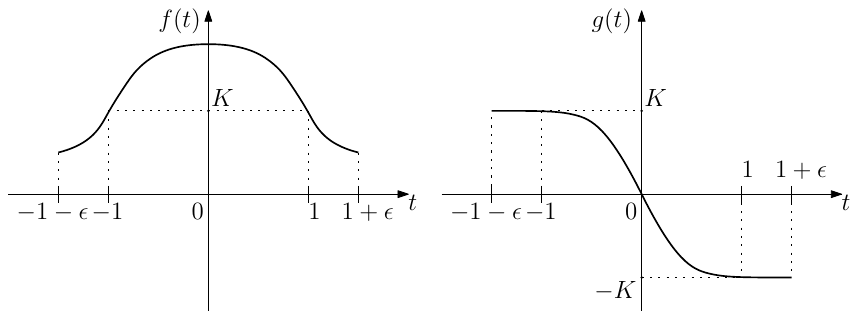}
	\caption{Smooth functions on the interval $(-1-\epsilon,1+\epsilon)$.}
	\label{fig:Gluing_functions}
\end{figure}

Consider the $1$-form $\sigma$ on $\mathcal{M}(\Sigma_1\cup_Y \overline{\Sigma}_2,\tilde{\phi}_1 \circ \tilde{\phi}_2)$ defined by the following rules:
\begin{center}
$
\displaystyle{\sigma = \begin{cases}
\lambda_{\beta_1}+K\pi_1^*(d\theta) &\textnormal{on \qquad \;\;\;} \textrm{int}(\mathcal{M}(\Sigma_1,\phi_1))\\
\;\;\;f\alpha+g\,d\theta &\textnormal{on \;\;\;} (-1-\epsilon,1+\epsilon) \times Y \times S^1\\
 \lambda_{\beta_2}-K\pi_2^*(d\theta) &\textnormal{on \qquad \;\;\;} \textrm{int}(\overline{\mathcal{M}(\Sigma_2,\phi_2)})
\end{cases}}
$
\end{center}

Observe the conditions (1) and (2) imply that $\sigma$ is smooth. Since $\sigma_i$ is a positive contact form on each $\mathcal{M}(\Sigma_i,\phi_i)$, checking that $\sigma$ is a positive contact form on the interiors $\textrm{int}(\mathcal{M}(\Sigma_1,\phi_1))$, $\textrm{int}(\overline{\mathcal{M}(\Sigma_2,\phi_2)})$ is immediate. Moreover, as $\textrm{dim}(\Sigma_i)=2n$, a standard computation gives that on a slightly bigger middle piece $(-1-\epsilon,1+\epsilon) \times Y \times S^1$ we have
$$\sigma \wedge (d\sigma)^n=nf^{n-1}(f'g-fg') \,dt \wedge \alpha \wedge (d\alpha)^{n-1}\wedge d\theta$$
which implies, by the condition (3), that $\sigma$ is a positive contact form on the middle piece as well. Therefore, we obtain a co-oriented contact structure $\eta=\textrm{Ker}(\sigma)$ on the folded sum $\mathcal{M}(\Sigma_1\cup_Y \overline{\Sigma}_2,\tilde{\phi}_1 \circ \tilde{\phi}_2)$.

\medskip
For the last claim, first for any $\theta \in S^1$ the $2$-form $d\sigma$ restricts to exact symplectic structures on the fibers $\pi_1^{-1}(\theta)$ and $\overline{\pi_2^{-1}(\theta)}$. (This is because $d\sigma$ restricts to $d\lambda_{\beta_i}$ on each $\mathcal{M}(\Sigma_i,\phi_i)$ and so Theorem \ref{thm:Contact_form_mapping _tori} applies.) Secondly, observe that for any $\theta \in S^1$ the restriction of $d(f\alpha+g d\theta)$ to the tangent spaces of the region $A_\theta:=(-1-\epsilon,1+\epsilon) \times Y \times \{\theta \}$ is nondegenerate everywhere except when $t=0$. To see this, we compute $$d(f\alpha+g d\theta)|_{TA_\theta}=(f'dt \wedge \alpha+fd\alpha+g'dt \wedge d\theta)|_{TA_\theta}=f'dt \wedge \alpha+fd\alpha$$ from which we have
$$[d(f\alpha+g d\theta)|_{TA_\theta}]^n=[f'dt \wedge \alpha+fd\alpha]^n=nf'f^{n-1}dt \wedge \alpha \wedge (d\alpha)^{n-1}+f^n(d\alpha)^{n}.$$ When we consider this as a top form on $TA_\theta$, the last term vanishes since $\alpha \in \Omega^1(Y)$ and $\textrm{dim}(Y)=2n-1$ (or, alternatively, $d\alpha$ has the full rank equal to $2n-2$ on $\xi \subset TY$). As a result, we conclude that on the tangent spaces of $A_\theta$, we have $$[d(f\alpha+g d\theta)|_{TA_\theta}]^n=nf'f^{n-1}dt \wedge \alpha \wedge (d\alpha)^{n-1}$$ which is a positive volume form everywhere except when $f'=0$ ($\Leftrightarrow t=0$ by the condition (4)).
Hence, except the points on the ``\textit{fold}'' $\{0\}\times Y \times \{\theta \}$, the form $d\sigma$ is nondegenerate on each fiber $\pi^{-1}(\theta)$, that is, the contact structure $\eta$ is compatible with the fibration $\pi$.

\end{proof}

\begin{definition} The contact form $\sigma$ on the folded sum $\mathcal{M}(\Sigma_1\cup_Y \overline{\Sigma}_2,\tilde{\phi}_1 \circ \tilde{\phi}_2)$ constructed in the above proof will be called a \textit{folded contact form}, and its kernel $\eta$ will be called a \textit{folded contact structure}.
\end{definition}

\begin{remark}
The construction of the folded contact form $\sigma$ (and so the structure $\eta$) on the folded sum $\pi: \mathcal{M}(\Sigma_1\cup_Y \overline{\Sigma}_2,\tilde{\phi}_1 \circ \tilde{\phi}_2) \longrightarrow S^1$ depends on the choice of a pair $(f,g)$ of smooth functions satisfying the conditions $(1)-(4)$. So one should write $\sigma(f,g)$ for $\sigma$ and $\eta(f,g)$ for $\eta$ and understand the effect of replacing $(f,g)$ with any other pair of smooth functions satisfying the same conditions. If we alter $(f_0,g_0)=(f,g)$ to another pair $(f_1,g_1)$ through the family $\{(f_s,g_s)\, |\, s\in [0,1]\}$ for which the conditions $(1)-(4)$ are satisfied (for all $s$), then (by following the steps in the above proof) one easily checks that $\{\sigma(f_s,g_s)\, |\, s\in [0,1]\}$ is a family of contact $1$-forms on $\mathcal{M}(\Sigma_1\cup_Y \overline{\Sigma}_2,\tilde{\phi}_1 \circ \tilde{\phi}_2)$, and so the family $\{\eta(f_s,g_s)\, |\, s\in [0,1]\}$ of their kernels defines an isotopy of contact structures compatible with the fibration map $\pi$. Hence, we may drop the pair $(f,g)$ from the notation unless it is necessary to emphases which functions are used in defining $\sigma$ (and so $\eta$).
\end{remark}

\begin{example}
Consider the \textit{trivial} (or \textit{standard}) contact mapping torus $$\pi_{st}:\mathcal{M}_{st}:=\mathcal{M}(D^{2n},id_{D^{2n}})=D^{2n}\times S^1 \longrightarrow S^1 \quad (\textrm{the projection onto} \;\; S^1)$$ with the trivial structure group \;$\textrm{Exact}(D^{2n},\partial D^{2n}, d\lambda_{st})=\{Id_{D^{2n}}\}$. The fiber $(D^{2n}, d\lambda_{st})$ of $\pi_{st}$ is the standard exact symplectic (indeed Stein) filling of the standard (tight) contact sphere $(S^{2n-1},\xi_{st}=\textrm{Ker}(\lambda_{st}))$ where $\lambda_{st}$ is the \textit{standard Liouville form} which is written as $\lambda_{st}=\sum_i(x_i dy_i - y_i dx_i)$ in the standard coordinates $x_1,y_1, ..., x_n,y_n$ on $D^{2n}\subset \mathbb{C}^n$. It is a well-known fact that (or by Theorem \ref{thm:Contact_form_mapping _tori}) we have the \textit{standard} (product) bundle contact form $\sigma_{st}=\lambda_{st}+\pi_{st}^*(d\theta)=\lambda_{st}+d\theta$ on $\mathcal{M}_{st}$ which defines the \textit{standard} bundle contact structure $\eta_{st}=\textrm{Ker}(\sigma_{st})$ which is, clearly, tight by Remark \ref{rmk:new_and_old_structures_isotopic}.

\medskip
Now, taking the folded sum of two copies of $\pi_{st}:\mathcal{M}_{st}\longrightarrow S^1$ (with its standard contact form $\sigma_{st}=\lambda_{st}+d\theta$) gives a fiber bundle $$\pi: \mathcal{M}:=\mathcal{M}(\,D^{2n}\,\cup_{S^{2n-1}} \overline{D^{2n}},\widetilde{Id}_{D^{2n}}\circ \widetilde{Id}_{D^{2n}}) \longrightarrow S^1$$ whose total space $$\mathcal{M}=S^{2n}\times S^1=D^{2n}\times S^1\bigcup_{S^{2n-1}\times S^1} D^{2n}\times S^1$$ admits a folded contact structure $\eta=\textrm{Ker}(\sigma)$ as in Theorem \ref{thm:Contact_Structure_on_Folded_Sum}. We note that $\eta$ is, in fact, Stein fillable (and hence tight) because $(S^{2n}\times S^1,\eta)$ is the convex boundary of the Stein domain $D^{2n+1}\times S^1$ obtained by attaching a Stein $1$-handle $D^1 \times D^{2n+1}$ to the Stein $0$-handle $D^{2n+2}$.   
\end{example}

%
%

\section{Tight Contact Structures on Folded Sums of Mapping Tori} \label{sec:Tight_Contact_Structures_on_Folded_Sums}

A contact structure $\xi$ on $M$ is called \emph{overtwisted} if there is an embedded overtwisted disk in $(M,\xi)$ (see \cite{E}, \cite{BEM}), otherwise it is called \emph{tight}. In this section, we will show that the folded sum operation produces tight contact structures under a suitable assumption on the common convex fold (Theorem \ref{thm:Tight_Parts_on_Folded_Sums_Reeb}). The proof will make use of ``contact vector fields''. A vector field $Z$ on a contact manifold $(M,\xi=\textrm{Ker}(\alpha))$ is said to be \emph{contact} if its flow preserves the contact distribution $\xi$, equivalently, if $\mathcal{L}_Z\alpha=f\alpha$ for some smooth function $f:M\to \mathbb{R}$. The following fundamental lemma in contact geometry characterizes contact vector fields on a given contact manifold. See \cite{Ge} for more details.

\begin{lemma} \label{lem:contact_vector_field}
Let $(M,\emph{Ker}(\alpha))$ be any contact manifold and $R_{\alpha}$ denote the ``Reeb vector field'' of $\alpha$ (i.e., $\alpha(R_\alpha)=1, \; \iota_{R_{\alpha}}d\alpha=0$). Then there is a one-to-one correspondence $$\{Z \in \Gamma(M)\,|\; Z \; \emph{ is contact}\}\longleftrightarrow \{H : M \to\R\,|\; H\; \emph{ is smooth}\}$$ between the set  of all contact vector fields on $M$ and the set of all smooth functions on $M$. The correspondence is given by $Z \to H_Z := \alpha(Z)$ ($H_Z$ is called the ``contact Hamiltonian'' of the contact vector field $Z$), and $H \to Z_H$ where $Z_H$ is the contact vector field uniquely determined by the equations \; $\alpha(Z_H)=H$ \; and \; $\iota_{Z_H} d\alpha=dH(R_{\alpha})\alpha-dH$. \qed
\end{lemma}



For the tightness result, one needs to assume that the contact mapping tori have Weinstein fibers whose convex boundaries satisfying certain condition. A \emph{Weinstein manifold} $(\Sigma,\omega,X,\Psi)$ is a symplectic manifold $(\Sigma,\omega)$ which admits a $\omega$-convex Morse function $\Psi:\Sigma \to \R$ whose complete gradient-like Liouville vector field is $X$ (see \cite{CE}). The following topologically characterizes Weinstein domains and will be useful.

\begin{theorem} [Lemma 11.13 in \cite{CE}] \label{thm:Top_Charac_of_Weinstein_manifolds}
Any Weinstein domain $(\Sigma,\omega,X,\Psi)$ of dimension $2n$ admits a (Weinstein) handle decomposition whose handles have indices at most $n$, and so, its core $\emph{Core}(\Sigma,\omega,X,\Psi)$ is a compact isotropic subcomplex of $(\Sigma,\omega)$ of dimension $n$. 
\end{theorem}

We prove:

\begin{theorem} \label{thm:Tight_Parts_on_Folded_Sums_Reeb}
For $n\geq 3$, suppose $(\Sigma_i^{2n}, d\beta_i,X_i,\Psi_i)$ are two Weinstein  domains with the same convex boundary $(Y,\emph{Ker}(\beta_i|_{TY}))$ such that the coinciding induced contact forms $\beta_1|_{TY}=\beta_2|_{TY}$ have no contractible periodic Reeb orbit. Let $\pi_i:\mathcal{M}_i:=\mathcal{M}(\Sigma_i,\phi_i)\longrightarrow S^1$ be any two contact mapping tori with the structure groups $\emph{Exact}(\Sigma_i,\partial \Sigma_i, d\beta_i)$, and consider their folded sum $\pi: \mathcal{M}:= \mathcal{M}(\Sigma_1\cup_Y \overline{\Sigma}_2,\tilde{\phi}_1 \circ \tilde{\phi}_2) \longrightarrow S^1$. Also let $\eta_i=\emph{Ker}(\sigma_i)$ be a bundle contact structure on each $\mathcal{M}_i$. Then the folded contact structure $\eta(f,g)=\emph{Ker}(\sigma(f,g))$ on $\mathcal{M}$ obtained from $\eta_1$ and $\eta_2$ is tight.
\end{theorem}

\begin{proof} We start with similar steps used in the proof of Theorem \ref{thm:Contact_Structure_on_Folded_Sum}: By assumption, we can write the induced contact structure as $$\xi=\textrm{Ker}(\alpha) \quad \textrm{ where } \quad \beta_1|_{TY}=\alpha=\beta_2|_{TY},$$ and taking $K>0$ is a sufficiently large, we write $\sigma_i=\lambda_{\beta_i}+K\pi_i^*(d\theta)$. Then by replacing $\alpha$ with $(1/K)\alpha$ (and modifying $\beta_1$ and $\beta_2$ accordingly near the common boundary $Y$), we obtain the folded contact form $\sigma(f,g)$ on $$\mathcal{M}=\mathcal{M}_1\bigcup_{\{-1\}\times Y \times S^1} [-1,1] \times Y \times S^1 \bigcup_{\{1\}\times Y \times S^1} \overline{\mathcal{M}_2}$$ given by the description
\begin{center}
$
\displaystyle{\sigma(f,g) = \begin{cases}
\lambda_{\beta_1}+K\pi_1^*(d\theta) &\textnormal{on \qquad \;\quad\quad} \textrm{int}(\mathcal{M}_1)\\
\;\;\;f\alpha+g\,d\theta &\textnormal{on \;\;\;} (-1-\epsilon,1+\epsilon) \times Y \times S^1\\
 \lambda_{\beta_2}-K\pi_2^*(d\theta) &\textnormal{on \qquad \;\quad\quad} \textrm{int}(\overline{\mathcal{M}_2})
\end{cases}}
$
\end{center}
where $\epsilon>0$ is sufficiently small so that
$\beta_1=Ke^{t+1}\alpha=\beta_2$ on $(-1-\epsilon,-1] \times Y$.
The Reeb vector field of $\sigma:=\sigma(f,g)$ is given by
\begin{center}
$
\displaystyle{R_\sigma= \begin{cases}
\qquad \qquad (1/K)\partial_{\theta} &\textnormal{on \qquad \;\quad\quad} \textrm{int}(\mathcal{M}_1)\\
\dfrac{-g'}{f'g-fg'}R_{\alpha}+\dfrac{f'}{f'g-fg'}\partial_{\theta} &\textnormal{on \;\;\;} (-1-\epsilon,1+\epsilon) \times Y \times S^1\\
 \; \quad \qquad -(1/K)\partial_{\theta}  &\textnormal{on \qquad \;\quad\quad} \textrm{int}(\overline{\mathcal{M}_2})
\end{cases}}
$
\end{center}
where $R_{\alpha}$ is the Reeb vector field of the contact form $\alpha$ on $Y$. Note that by assumption $R_{\alpha}$ has no contractible periodic orbit.

\medskip
Now, on the contrary, suppose that $(\mathcal{M},\eta(f,g))$ is overtwisted. We know from \cite{CMP} that having an overtwisted disk is equivalent to having a plastikstufe. On the other hand, the existence of a plastikstufe in a closed contact manifold implies the existence of a contractible periodic Reeb orbit associated to every contact form defining the contact structure \cite{AH}. So let $\reeb\approx S^{1}$ be an embedded contractible periodic Reeb orbit in $(\mathcal{M},\eta(f,g))$ associated to the contact form $\sigma$. In what follows we will use: 1) Any transverse knot stays transverse under the flows of contact vector fields. 2) In a small neighborhood of it, any transverse knot can be made parallel to the Reeb directions (i.e., can be realized as a Reeb orbit) by multiplying the original contact form with a suitable smooth function whose support lies in the neighborhood. (See \cite{Ge} for more details.) Hence, whenever we move $\reeb$ under such flows, we may assume that its image is still a Reeb orbit of the new contact form (isotopic to the older one).

\medskip
Since $\reeb$ bounds a $2$-disk, say $\dd$, by pushing them together along the flow of a contact vector field whose contact Hamiltonian is obtained by multiplying the contact Hamiltonian $H_\sigma:\mathcal{M} \to \R$ (constant $1$ map) of $R_\sigma$ with a suitable smooth function $\mathcal{M} \to \mathbb{R}$, we may assume that $\dd$ (and so its boundary $\reeb$) misses a fiber $\pi^{-1}(\theta_0)$ over some $\theta_0 \in S^1$. So
$$\reeb \subset \dd \subset \mathcal{M}\setminus \pi^{-1}(\theta_0)\cong \left(\Sigma_1 \cup_Y \overline{\Sigma}_2\right) \times \mathbb{R}.$$
Note that by replacing the $\theta$-coordinate on $S^1$ by the $z$-coordinate on $\mathbb{R} \cong S^1 \setminus \{\theta_0\}$, one obtains the restriction of any function, form or vector field on $\mathcal{M}$ to any subset of $\mathcal{M}\setminus \pi^{-1}(\theta_0)$.
The next step of the proof will rely on the following two general facts.

\begin{lemma} \label{lem:transverse}
Suppose that $n\geq 2$. Let $L$ be any transverse knot in a contact manifold $(M,\xi)$ of dimension $2n+1$. Also let $\Delta$ be any compact subset of $M$ admitting a CW-complex structure of finite type consisting of cells of dimension $\leq n+1$ such that $L\not\subset \Delta$. Then $L$ is contact isotopic (through transverse knots) to a transverse knot $L'$ in any given $\epsilon$-neighborhood of itself so that $L' \cap \Delta=\emptyset$.
\end{lemma}

\begin{proof}
It is well-known that any transverse knot has a neighborhood contactomorphic to the standard model (see Example 2.5.16 in \cite{Ge}, for instance). More precisely, there exists an $\epsilon$-neighborhood $N_\epsilon$ of $L$ in $M$ (which can be taken arbitrarly small) such that $(N_\epsilon,\xi |_{N_\epsilon})$ is contactomorphic to $$\left(S^1 \times \R^{2n},\textrm{Ker}(d\theta+ \sum_{i=1}^n (x_i dy_i-y_idx_i))\right)$$ where $L$ corresponds to $S^1 \times \{0\}$. Since $L\not\subset \Delta$, there is a point, say $p_0$, in $L\setminus \Delta$ which corresponds to a point $\theta_0 \times \{0\}\in S^1 \times \R^{2n}$. If $L\cap \Delta=\emptyset$, then nothing to prove. If not,
$$L\cap \Delta \subset \left(\R \times \R^{2n},\textrm{Ker}(dz+ \sum_{i=1}^n (x_i dy_i-y_idx_i))\right)\stackrel{\Theta}{\cong} \left(\R \times \R^{2n},\textrm{Ker}(dz+ \sum_{i=1}^n x_i dy_i)\right)$$ where $z$ is the coordinate on $\R=S^1 \setminus \{\theta_0\}$ and $\Theta$ is the well-known contactomorphism. (The explicit rule of $\Theta$ can be found, for instance, in Example 2.1.3 in \cite{Ge}.) Therefore, in this (the most right) local model near $L\setminus \{p_0\}$, one can consider $2n+1$ linearly independent smooth contact vector fields $Z_1, Z_2, ..., Z_{2n+1}$ given by
\begin{equation} \label{eqn:contact_vector_fields}
	Z_1=\partial_{y_1}\;,...,\;Z_n=\partial_{y_n}, \quad Z_{n+1}=\partial_{x_1}-y_1 \partial_{z}\;,...,\;Z_{2n}=\partial_{x_n}-y_n \partial_{z}\;, \quad Z_{2n+1}=\partial_{z}.
\end{equation}

The corresponding contact Hamiltonian functions (as in Lemma \ref{lem:contact_vector_field}), respectively, are
\begin{equation} \label{eqn:contact_Hamiltonians}
	H_1=x_1\;,...,\;H_n=x_n, \quad H_{n+1}=-y_1\;,...,\;H_{2n}=-y_n\;, \quad H_{2n+1}=1.
\end{equation}

For $p \in \Delta$, let $d_{max}(\Delta,p)$ denote the maximum of the dimensions of the cells in $\Delta$ containing $p$. Clearly, $d_{max}(\Delta,p)\leq n+1$ for all $p \in \Delta$. So, for any $p \in L\cap \Delta$ we have $$dim(T_pM)-[d_{max}(\Delta,p)+dim(T_pL)]\geq 2n+1-[n+1+1]=n-1\geq 1 \quad \textrm{(as } n\geq 2).$$ 

\noindent Therefore, at each $p \in L\cap \Delta$, there exists at least one direction in $T_pM$ which is transversal to both $\Delta$ and $L$ at $p$. Any such direction can be written as a linear combination of $Z_1, Z_2, ..., Z_{2n+1}$. Thus, by compactness of $L$ and $\Delta$ (and also by finiteness of $\Delta$), there exist smooth functions $\mu_1, \mu_2, ..., \mu_{2n+1}: \R \times\R^{2n} \to \mathbb{R}$ compactly supported near $L\cap \Delta$ [i.e., in $(\R \times\R^{2n},\textrm{Ker}(dz+ \sum_{i=1}^n x_i dy_i))$] such that by pushing $L$ using the flow of the contact vector field corresponding to the contact Hamiltonian $\oplus_{i=1}^{2n+1}\mu_i H_{i}$, one obtains a transverse knot $L'$ (contact isotopic to $L$ through transverse knots) disjoint from $\Delta$.

\end{proof}

\begin{lemma} \label{lem:contractible_transverse}
Suppose that $n\geq 3$. Let $L$ be any transverse knot bounding a $2$-disk $D$ in a contact manifold $(M,\xi)$ of dimension $2n+1$. Also let $\Delta$ be any compact subset of $M$ admitting a CW-complex structure of finite type consisting of cells of dimension $\leq k$ such that $L\not\subset \Delta$. Then if $k\leq n+1$, $L$ is contact isotopic (through transverse knots) to a transverse knot $L'$ in any given $\epsilon$-neighborhood of itself so that $L' \cap \Delta=\emptyset$. Furthermore, this contact isotopy extends to a $\epsilon$-small homotopy from $D$ to another $2$-disk $D'$ in $M$ so that $\partial D'=L'$ and $D' \cap \Delta=\emptyset$.
\end{lemma}

\begin{proof}
First note that the condition $dim(\Delta)\leq n+1$ (the dimension assumption in Lemma \ref{lem:transverse}) is satisfied, so the claim about the existence of the contact isotopy $\Phi: S^1 \times [0,1] \to M$ from $L=\Phi(s,0)$ to another transverse knot $L'=\Phi(s,1)$ (with $L' \cap \Delta=\emptyset$) follows directly from Lemma \ref{lem:transverse}.  

\medskip
\noindent As before, if $D\cap \Delta=\emptyset$, then nothing to prove. If not, observe this time $d_{max}(\Delta,p)\leq k$ for all $p \in \Delta$ by assumption. So, for any point $p \in L\cap \Delta$ we have 
$$dim(T_pM)-[d_{max}(\Delta,p)+dim(T_pD)]\geq 2n+1-[k+2]=2n-k-1.$$
Note that since $n\geq 3$ and $k\leq n+1$ we have $2n-k-1\geq 1$. Therefore, at each $p \in D\cap \Delta$, there exists at least one direction in $T_pM$ which is transversal to both $\Delta$ and $D$ at $p$. Indeed, when $p \in L\cap \Delta \subset D\cap \Delta$ the corresponding directions are already used to obtain the contact isotopy $\Phi$ from $L=\Phi(s,0)$ to $L'=\Phi(s,1)$. So let $\widetilde{\Phi}: D^2 \times [0,1] \to M$ be any homotopy extension of $\Phi$ obtained by moving $D$ along the directions transversal to both $\Delta$ and $D$. Clearly, such homotopy can be achieved in an arbitrary small neighborhood of $D$. By construction $D=\widetilde{\Phi}(s,0)$, $D'\:=\widetilde{\Phi}(s,1)$, $L'=\partial D'$ and $D' \cap \Delta=\emptyset$ as required.

\end{proof}

Returning back to the proof,
let $\Sigma^z_1:=\pi_1^{-1}(z)(\cong \Sigma_1)$ and $\Sigma^z_2:=\pi_2^{-1}(z)(\cong \overline{\Sigma}_2)$ denote the fiber of $\pi_1$ and $\pi_2$ over a point $z \in \R \cong S^1 \setminus \{\theta_0\}$ equipped with the Weinstein structures $(d\beta^z_1,X^z_1,\Psi^z_1)$ and $(d\beta^z_2,X^z_2,\Psi^z_2)$, respectively. We know from Theorem \ref{thm:Top_Charac_of_Weinstein_manifolds} that the core $\textrm{Core}(\Sigma^z_i,d\beta^z_i,X^z_i,\Psi^z_i)$ of each of these Weinstein domains is a compact $n$-dimensional subcomplex. Note that the union of these cores gives the ``core'' $\textrm{Core}(\mathcal{M}_i\setminus \pi_i^{-1}(\theta_0))$ of the trivial fibration $$\pi_i|_{\mathcal{M}_i\setminus \pi_i^{-1}(\theta_0)}:\Sigma_i \times \R \to \R$$ which is just a $1$-dimensional thickening of $\mathcal{C}_i:=\textrm{Core}(\Sigma_i,d\beta_i,X_i,\Psi_i)$, i.e.,
\begin{center}
$\textrm{Core}(\mathcal{M}_i\setminus \pi_i^{-1}(\theta_0))=\bigcup_{z \in \R}\textrm{Core}(\Sigma^z_i,d\beta^z_i,X^z_i,\Psi^z_i)\cong \mathcal{C}_i\times \R.$
\end{center}
Putting these cores together we get the core $(\mathcal{C}_1\cup \mathcal{C}_2) \times \R$  of $\mathcal{M}\setminus \pi^{-1}(\theta_0)\cong \left(\Sigma_1 \cup_Y \overline{\Sigma}_2\right) \times \mathbb{R}$.
Since $\dd \cap \pi^{-1}(\theta_0)=\emptyset$, the compactness of $\dd$ implies that there exists $0<\delta\in \R$ such that
\begin{center}
$\reeb \subset \dd \subset \left(\left(\Sigma_1 \cup_Y \overline{\Sigma}_2\right) \times [-\delta,\delta],\eta(f,g)|_{\left(\Sigma_1 \cup_Y \overline{\Sigma}_2\right) \times [-\delta,\delta]} \right).$
\end{center}
Consider the core $\mathcal{C}:=(\mathcal{C}_1\cup \mathcal{C}_2) \times [-\delta,\delta]$ of $\left(\Sigma_1 \cup_Y \overline{\Sigma}_2\right) \times [-\delta,\delta]$. By assumption, $\mathcal{C}$ is a finite compact subcomplex consisting of cells of dimension at most $\leq n+1$. Therefore, by Lemma \ref{lem:contractible_transverse}, we may assume that $\reeb$ (and the $2$-disk $\dd$ it bounds) are disjoint from the core $\mathcal{C}$, and hence from the core $(\mathcal{C}_1\cup \mathcal{C}_2) \times \R$.

\medskip
Next using the projection $\Sigma_1 \times \R \to \Sigma_1^z$ (resp., $\overline{\Sigma}_2 \times \R \to \Sigma_2^z$), by pulling back the Liouville vector field $X^z_1$ (resp., $X^z_2$) of a fiber $\Sigma_1^z$ (resp., $\Sigma_2^z$) obtain a smooth vector field $X_1$ (resp., $X_2$) on $\mathcal{M}_{1}\setminus \pi_1^{-1}(\theta_0)$ (resp,. on $\overline{\mathcal{M}_{2}}\setminus \pi_2^{-1}(\theta_0)$). Note that the restrictions of $X_i$ to each fiber $\Sigma_i^z$ is the Liouville vector field of that fiber by construction. Therefore, for each point $p \in \Sigma_i^z \setminus \textrm{Core}(\Sigma_i^z)$, there is a unique flow line of $X_i$ connecting $p$ to a unique point $q$ in the convex boundary $\partial \Sigma_i^z$ (a copy of $Y$). 
Consider the smooth vector field $Z$ on the contact open subdomain 
$\left(\mathcal{M}\setminus \pi^{-1}(\theta_0), \eta(f,g)|_{\mathcal{M}\setminus \pi^{-1}(\theta_0)} \right)$ given by:
\begin{center}
$
\displaystyle{Z= \begin{cases}
\qquad \;\; X_1 + z\partial_{z} &\textnormal{on \qquad \;\quad} \textrm{int}(\mathcal{M}_{1}\setminus \pi^{-1}(\theta_0))\\
\dfrac{fg}{f'g-fg'}\partial_{t} + z\partial_{z} &\textnormal{on \;\;\;} (-1-\epsilon,1+\epsilon) \times Y \times \R \\
\qquad \;\; X_2 + z\partial_{z} &\textnormal{on \qquad \;\quad} \textrm{int}(\overline{\mathcal{M}_{2}}\setminus \pi^{-1}(\theta_0))
\end{cases}}
$
\end{center}
Again the smoothness of $Z$ follows from the conditions $(1)-(4)$ on $f$ and $g$, and one can easily check that $Z$ is contact and its contact Hamiltonian function is $H_Z=z$.

\medskip
Let $Z'$ be the contact vector field whose contact Hamiltonian is $H_{Z'}=\mu H_Z$ where $\mu:\mathcal{M}\setminus \pi^{-1}(\theta_0) \to \R$ is a smooth function satisfying:
\begin{itemize}
\item $\mu\equiv 1$ \quad on \quad
$\displaystyle [\mathcal{M}\setminus \pi^{-1}(\theta_0)] \setminus ([-1,1] \times Y \times \R)$,\\
\item $\mu(t,x,z)=h(t)$ on
 $[-1,1] \times Y \times \R$,  where $h(t)$ is a smooth function on $[-1,1]$ such that $h$ is even, $h\equiv 1$ near $t=\pm 1$ and $h(t)=0$ if and only if $t=0$ (see Figure \ref{fig:the_function_h}).
\end{itemize}
\begin{figure}[h]
	\centering
	\includegraphics[scale=.85]{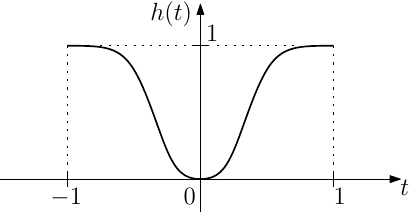}
	\caption{Smooth function $h(t)$ on the interval $[-1,1])$.}
	\label{fig:the_function_h}
\end{figure}
By pushing (isotoping) $\dd$ (and so $\reeb$) along the flow of $Z'$, eventually their final images will be completely contained in $[-1,1] \times Y \times \R$. Assuming this has been already done, we may assume that the contact manifold $$\left( [-1,1] \times Y \times \R, \textrm{Ker}(f(t)\alpha+g(t)dz) \right)$$ contains a periodic Reeb orbit $\reeb$ and the disk $\dd$ it bounds.

\medskip
Observe that under the flow maps of $Z=fg/(f'g-fg')\partial_{t} + z\partial_{z}$, the $z$-coordinates of the images of the points in $[-1,1] \times Y \times \{0\}$ are stationary (and all equal to $0$). In other words, the flow lines entering $(z=0)$-level set $[-1,1] \times Y \times \{0\}$ can not escape the set. By the compactness, after sufficiently isotoping $\dd$ along the flow lines of $Z$, we may assume that $\dd$ (and so $\reeb$) is completely contained in $[-1,1] \times Y \times \{0\}$. Now note that 
\begin{center}
$\dfrac{fg}{(f'g-fg')}>0 \;\; \textrm{if}\;\; t\in [-1,0),\; \dfrac{fg}{(f'g-fg')}=0 \;\;\textrm{if}\;\; t=0,\; \dfrac{fg}{(f'g-fg')}<0 \;\;\textrm{if}\;\;  t\in (0,1].$
\end{center}
Consequently, under the flow maps of $Z$, the $t$-coordinates of the images of the points in $\{0\} \times Y \times \{0\}$ are stationary (and all equal to $0$), and so, the flow lines entering $(t=0)$-level set $\{0\} \times Y \times \{0\}$ can not escape the set. Again, by the compactness, after sufficiently pushing $\dd$ along the flow lines of $Z$, eventually the image of $\dd$ (and so $\reeb$) is completely contained in $\{0\} \times Y \times \{0\}$. Denote this final image of $\reeb$ in $\{0\} \times Y \times \{0\}$ by $\reeb'$.

\medskip
Finally, since $f(0)=Ke$ and $g(0)=0$, we have  $\sigma(f,g)|_{\{0\}\times Y \times \{0\}}=Ke\alpha$. Therefore, $(\{0\}\times Y \times \{0\},\eta(f,g)|_{\{0\}\times Y \times \{0\}})$ is contactomorphic to $(Y,\xi=\textrm{Ker}(\alpha))$. Indeed, this is a strict contactomorphism upto a rescaling of $\alpha$ via the constant $Ke$, so the Reeb orbits of $Ke\alpha$ are just those of $\alpha$ reparametrized with a constant scale of $1/Ke$. Thus, we can find an isotopy of contact forms (compactly supported in an $\epsilon$-neighborhood of $\reeb'$ in $Y=\{0\}\times Y \times \{0\}$) from $\alpha$ to another contact form $\alpha'$ of $\xi$ such that $\reeb'$ is parallel to the Reeb directions determined by $\alpha'$. However, an existence of a contractible periodic Reeb orbit in $(Y,\xi=\textrm{Ker}(\alpha'))$ is a direct contradiction to the assumption stated as $\alpha$ and $\alpha'$ have the same (isomorphic) Reeb dynamics. This finishes the proof of Theorem \ref{thm:Tight_Parts_on_Folded_Sums_Reeb}.

\end{proof}


\begin{remark}
It is well known that any contact structure on a closed manifold is supported by an open book with Weinstein pages (see, for instance, \cite{Gi2}, \cite{Ge}). This is one of the motivations that the fibers are assumed to be Weinstein domains in Threorem \ref{thm:Tight_Parts_on_Folded_Sums_Reeb}. However, using almost the same steps, one can prove a similar statement in the case where fibers are Liouville domains (of finite types) with at most half-dimensional cores. Of course, one must still assume that coinciding contact forms on the common boundary have no contractible periodic Reeb orbits. 
\end{remark}

\begin{example}
Let us consider the unit cotangent disk bundle $D^*T^3\cong T^3\times D^3$ of the $3$-torus which can be equipped with a Liouville structure as follows: Let $(\theta_1,\theta_2,\theta_3)$ denote the (global) angular coordinates on $T^3$ and  $(x_1,x_2,x_3)$ denote the standard euclidean coordinates on $D^3$. So  $$D^*T^3=\{(\theta_1,\theta_2,\theta_3,x_1,x_2,x_3) \,|\, \theta_i \in [0,2\pi], \; 0\leq x_1^2+x_2^2+x_3^2\leq1 \}.$$ Then for the $1$-form $\beta=x_1 d\theta_1-x_2d\theta_2+x_3d\theta_3$, one easily see that $(d\beta)^3$ is a volume form on $D^*T^3$. Also it is straight forward to check that the vector field $$X=x_1\dfrac{\partial}{\partial x_1}+x_2\dfrac{\partial}{\partial x_2}+x_3\dfrac{\partial}{\partial x_3}$$ is a Liouville field of $d\beta$, and it is transversely pointing out from the boundary $$\partial(D^*T^3)=T^3\times S^2=\{(\theta_1,\theta_2,\theta_3,x_1,x_2,x_3) \,|\, \theta_i \in [0,2\pi], \; x_1^2+x_2^2+x_3^2=1 \}.$$ Clearly, the flow of $X$ exists for all time and its backward flow shrinks $D^*T^3$ onto its core $T^3\times\{pt\}$. Therefore,  $(D^*T^3,d\beta, X)$ is a $6$-dimensional Liouville domain with a compact core of dimension $3$ (half-dimension).

\medskip
The Reeb vector field of the induced contact form $\alpha:=\beta|_{T(T^3\times S^2)}$ on the convex boundary $(T^3\times S^2, \textrm{Ker}(\alpha))$ is given by $$R=x_1\dfrac{\partial}{\partial \theta_1}-x_2\dfrac{\partial}{\partial \theta_2}+x_3\dfrac{\partial}{\partial \theta_3}.$$ Clearly, all periodic Reeb orbits travels along the compositions of the directions along the coordinates $(\theta_1,\theta_2,\theta_3)$ which are homotopically non-trivial in $T^3\times S^2$, and hence, the contact form $\alpha$ has no contractible periodic Reeb orbit. So, the related conditions in Theorem \ref{thm:Tight_Parts_on_Folded_Sums_Reeb} are all satisfied. 

\medskip
Now fix any two self-diffeomorphisms $$\phi_1, \phi_2 \in \textrm{Exact}(D^*T^3,\partial (D^*T^3), d\beta)=\textrm{Exact}(T^3\times D^3,T^3\times S^2, d\beta),$$ and consider the contact mapping tori
$\pi_i:\mathcal{M}(T^3\times D^3,\phi_i)\longrightarrow S^1$ for $i=1,2$. Then, by Theorem \ref{thm:Tight_Parts_on_Folded_Sums_Reeb}, the folded contact structure $\eta$ on the total space of their folded sum 
$\pi: \mathcal{M}:= \mathcal{M}(T^3\times D^3\cup_{T^3\times S^2}\overline{T^3\times D^3},\tilde{\phi}_1 \circ \tilde{\phi}_2) \longrightarrow S^1$ is tight.

\medskip
For instance, if we take $\phi_1=\phi_2=Id_{T^3\times D^3}$, then the folded contact structure $\eta$ (constructed as in Theorem \ref{thm:Contact_Structure_on_Folded_Sum}) is a tight contact structure on the total space $$\mathcal{M}\cong T^3\times S^3 \times S^1$$ of the folded sum of two copies of the trivial mapping torus $\mathcal{M}(T^3\times D^3,Id_{T^3\times D^3})\longrightarrow S^1$. 
 
\end{example}


\noindent \textbf{Data availability:} Data sharing not applicable as no dataset was generated in this work.

\noindent \textbf{Competing Interests:} The author has no relevant financial or non-financial interests
to disclose.

\end{document}